\documentclass[12pt,reqno]{amsart}
\usepackage[margin=1in]{geometry}
\usepackage{amsxtra,amssymb,amsfonts,amsmath,mathrsfs,amscd,etoolbox,mathtools,booktabs,verbatim}
\usepackage[UKenglish]{isodate}
\mathtoolsset{showonlyrefs}
\usepackage[breaklinks=true,linktocpage=true,colorlinks=true,linkcolor=teal!75!blue,citecolor=teal!75!blue,filecolor=teal!75!blue,urlcolor=black!39!blue,backref=page]{hyperref}
\hypersetup{linktocpage}
\usepackage[hiresbb]{graphicx,xcolor}
\usepackage{orcidlink}
\usepackage{tikz}
\usepackage{tikz-cd}
\usetikzlibrary{shapes.geometric}
\usetikzlibrary{shapes.misc}
\usetikzlibrary{positioning}
\allowdisplaybreaks[4]

\newtheorem{theorem}{Theorem}[section]
\newtheorem{lemma}[theorem]{Lemma}

\theoremstyle{definition}

\numberwithin{equation}{section}

\newcommand{\C}{\mathbb{C}}
\newcommand{\F}{\mathbb{F}}

\newcommand{\Q}{\mathbb{Q}}
\newcommand{\R}{\mathbb{R}}
\newcommand{\Z}{\mathbb{Z}}
\newcommand{\SL}{\mathrm{SL}}
\newcommand{\PGL}{\mathrm{PGL}}

\renewcommand{\Re}{\mathrm{Re}}
\renewcommand{\tilde}{\widetilde}
\renewcommand{\epsilon}{\varepsilon}
\patchcmd{\section}{\scshape}{\bfseries}{}{}
\makeatletter
\renewcommand{\@secnumfont}{\bfseries}
\makeatother
\makeatletter\newcommand{\tpmod}[1]{{\@displayfalse \pmod{#1}}}
\newcommand{\MYhref}[3][black!39!blue]{\href{#2}{\color{#1}{#3}}}
\renewcommand*\backref[1]{$\uparrow$\thinspace\ifx#1\relax \else #1 \fi}

\begin{document}

\title[QUE for Eisenstein Series on Bruhat--Tits Buildings]{Quantum Unique Ergodicity for Eisenstein Series on Bruhat--Tits Buildings}

\author[Ikuya Kaneko]{Ikuya Kaneko\,\orcidlink{0000-0003-4518-1805}}
\address[Ikuya Kaneko]{The Division of Physics, Mathematics and Astronomy, California Institute of Technology, 1200 East California Boulevard, Pasadena, California 91125, United States of America}
\urladdr{\href{https://sites.google.com/view/ikuyakaneko/}{https://sites.google.com/view/ikuyakaneko/}}
\email{ikuyak@icloud.com}

\author[Shin-ya Koyama]{Shin-ya Koyama\,\orcidlink{0000-0002-8529-5244}}
\address[Shin-ya Koyama]{Department of Biomedical Engineering, Toyo University, 2100 Kujirai, Kawagoe-shi, Saitama 350-8585, Japan}
\urladdr{\href{http://www1.tmtv.ne.jp/~koyama/}{http://www1.tmtv.ne.jp/~koyama/}}
\email{koyama@tmtv.ne.jp}

\subjclass[2020]{\MYhref[black]{https://mathscinet.ams.org/mathscinet/msc/msc2020.html?t=&s=58j51}{58J51} (primary); \MYhref[black]{https://mathscinet.ams.org/mathscinet/msc/msc2020.html?t=&s=11f12}{11F12}, \MYhref[black]{https://mathscinet.ams.org/mathscinet/msc/msc2020.html?t=&s=81q50}{81Q50} (secondary)}


\cleanlookdateon

\date{\today}

\begin{abstract}
We prove the quantum unique ergodicity conjecture for Eisenstein series over function fields in the level aspect. Adapting the machinery of Luo--Sarnak (1995), we employ the spectral decomposition and handle the cuspidal and Eisenstein contributions separately.
\end{abstract}

\maketitle

\section{Introduction}
A famous conjecture of Berry~\cite{Berry1977} and Hejhal--Rackner~\cite{HejhalRackner1992} asserts that Hecke--Maa{\ss} newforms ought to behave like random waves in the limit of large Laplace eigenvalue.~This is formulated more rigorously by proposing that on any compact continuity set, the moments of Hecke--Maa{\ss} newforms of large Laplace eigenvalue ought to be asymptotically equal to the moments of a standard normal random variable. As for the second moment, the statement of the random wave conjecture is equivalent to the quantum unique ergodicity conjecture raised by Rudnick--Sarnak~\cite{RudnickSarnak1994}, which was resolved subsequently by Lindenstrauss~\cite{Lindenstrauss2006} and Soundararajan~\cite{Soundararajan2010} in the case of the modular surface. Prior to these breakthroughs, quantum unique ergodicity with an optimal rate of convergence was known to follow under the generalised Lindel\"{o}f hypothesis for certain triple product $L$-functions by Watson~\cite{Watson2002}.

Luo--Sarnak~\cite{LuoSarnak1995} was the first to consider quantum unique ergodicity for the continuous spectrum spanned by the Eisenstein series $E(z, s)$ on $\SL_{2}(\Z) \backslash \mathbb{H}$, where $\mathbb{H} \coloneqq \{z = x+iy \in \mathbb{C}: y > 0 \}$ stands for the upper half-plane. Let $\phi$ be a fixed continuous and compactly~supported function of level $1$. Among other results, they proved that
\begin{equation}\label{QUE}
\int_{\SL_{2}(\Z) \backslash \mathbb{H}} \phi(z) \left|E \left(z, \frac{1}{2}+it \right) \right|^{2} d\mu(z) 
\to \frac{6 \log t}{\pi} \int_{\SL_{2}(\Z) \backslash \mathbb{H}} \phi(z) d\mu(z), \qquad t \to \infty,
\end{equation}
where $d\mu(z) = y^{-2} dx dy$. We refer the interested reader to the work of Zhang~\cite{Zhang2019} and the first author~\cite{ChatzakosDarreyeKaneko2023} for generalisations of~\eqref{QUE} to higher rank groups. As is seen in~\cite{Sarnak1995}, it is beneficial to study the problem in various aspects in term of the parameters in functional equations of $L$-functions. Since then, we have studied the phenomenon in a different aspect; in particular, the authors~\cite{KanekoKoyama2020,Koyama2009} and Pan--Young~\cite{PanYoung2021} proved, with $p$ traversing primes, that
\begin{equation}\label{level}
\int_{\Gamma_{0}(p) \backslash \mathbb{H}} \phi(z) \left|E_{\infty} \left(z, \frac{1}{2}+it \right) \right|^{2} d\mu(z) 
\to \frac{6 \log p}{\pi} \int_{\SL_{2}(\Z) \backslash \mathbb{H}} \phi(z) d\mu(z), \qquad p \to \infty,
\end{equation}
where $E_{\infty}(z, s)$ is the Eisenstein series of weight zero, level $p$, and trivial central character.

On the other hand, the analogy between rational numbers and polynomials over finite~fields plays an essential role in number theory. If $q > 3$ is a prime, then the alternative~underlying group is $\PGL_{2}(\F_{q}[T])$ acting on $G = \PGL_{2}(\F_{q}((T^{-1})))$. We handle the congruence subgroup $\Gamma_{0}(A)$ with $A \in \F_{q}[T]$ a monic irreducible polynomial. The space $\Gamma_{0}(A) \backslash G$ can be realised as an arithmetic infinite graph for which the spectral theory of automorphic forms is developed by Nagoshi~\cite{Nagoshi2001}. Furthermore, the space is geometrically finite in the sense that it is a union of a finite graph with finitely many infinite rays, and the spectra of the corresponding Laplacian consist of a finite number of eigenvalues. Hence, there is no analogue of~\eqref{QUE} to be formulated in our case. This naturally leads us to the study of the level aspect version~\eqref{level} over function fields by taking a sequence of polynomials $A$ such that $\deg{A}$ tends to infinity.

We proceed to the statement of our main result. Let $K = \PGL_{2}(\F_{q}[[T^{-1}]])$ be a maximal compact subgroup of $G$, and let $\mathbb{H} = G/K$. Let $\Gamma_{\infty} \coloneqq \{\gamma \in \Gamma_{0}(A): (0, 1) \gamma = (0, 1) \}$ denote the stabiliser of $\infty$ in $\Gamma_{0}(A)$. The Eisenstein series $E(g, s)$ for $g \in \mathbb{H}$ is defined by
\begin{equation}\label{Eisenstein-series}
E(g, s) \coloneqq \sum_{\gamma \in \Gamma_{\infty} \backslash \Gamma_{0}(A)}
\left(\frac{|\det(\gamma g)|}{h((0, 1) \gamma g)^{2}} \right)^{s},
\end{equation}
where $h((x, y)) = \sup \{|x|, |y| \}$ for each vector $(x, y)$ in $k_{\infty} \times k_{\infty}$ with $k_{\infty} = \F_{q}((T^{-1}))$. It is absolutely convergent for $\Re(s) > 1$ and meromorphic in the entire complex plane $\mathbb{C}$.~Li~\cite{Li1979} and Nagoshi~\cite{Nagoshi2001} demonstrated that $E(g, \frac{1}{2}+it)$ corresponds to the continuous spectrum of the adjacency operator. A class in $\mathbb{H}$ is represented by a matrix of the shape
\begin{equation*}
g = \begin{pmatrix} T^{n} & x \\ 0 & 1 \end{pmatrix}
 = \begin{pmatrix} 1 & x \\ 0 & 1 \end{pmatrix} \begin{pmatrix} T^{n} & 0 \\ 0 & 1 \end{pmatrix},
\end{equation*}
where $n \in \Z$ and $x \in \F_{q}((T^{-1}))$. We define the measure $dg$ on $\Gamma_{0}(A) \backslash \mathbb{H}$ by
\begin{equation*}
\int \phi(g) dg = \sum_{n \in \Z} \int \phi \left(\begin{pmatrix} T^{n} & x \\ 0 & 1 \end{pmatrix} \right) dx,
\end{equation*}
where the measure $dx$ is normalised so that $\int_{\F_{q}((T^{-1}))/\F_{q}[T]} dx = 1$.
\begin{theorem}\label{main}
Keep the notation as above. Let $A \in \F_{q}[T]$ be a monic irreducible polynomial, and let $t \in \R^{\times}$. For a fixed continuous and compactly supported function $\phi$, we have that
\begin{equation*}
\int_{\Gamma_{0}(A) \backslash \mathbb{H}} \phi(g) \left|E \left(g, \frac{1}{2}+it \right) \right|^{2} dg
 = \frac{(1+q^{-1})}{2\log q} \log|A| \int_{\Gamma_{0}(1) \backslash \mathbb{H}} \phi(g) dg+O(1),
\end{equation*}
where the implicit constant is absolute and effectively computable.
\end{theorem}

\subsection*{Acknowledgements}
We extend our gratitude to Nobushige Kurokawa, Yuichiro Taguchi, and Masao Tsuzuki for their valuable feedback. We are also grateful to the anonymous referee for their detailed comments, which significantly enhanced the presentation of the~paper.

\section{Preliminaries}\label{preliminaries}
This section sets up notation and compiles a spectral toolbox including Eisenstein series on Bruhat--Tits buildings and characters of $\F_{q}((T^{-1}))$. We also compute the index of $\Gamma_{0}(A)$ in $\PGL_{2}(\F_{q}[T])$. Throughout this section, we denote $\Gamma = \Gamma_{0}(A)$ for notational convenience.


\subsection{Function fields}\label{function-fields}
Let $\F_{q}$ be the finite field, $\F_{q}[T]$ be the ring of polynomials in $T$~over $\F_{q}$, and $k = \F_{q}(T)$ be the quotient field. A function field $\F_{q}(T)$ behaves like number fields over $\Q$. Nonetheless, the infinite place in $\F_{q}(T)$ is a non-archimedean local field, while~the~infinite place of $\Q$ is archimedean. Let $G = \PGL_{2}(k_{\infty})$ and $K = \PGL_{2}(r_{\infty})$, where $k_{\infty} = \F_{q}((T^{-1}))$ is the field of Laurent series in the uniformiser $T^{-1}$ and $r_{\infty} = \F_{q}[[T^{-1}]]$ is the ring of Taylor series in $T^{-1}$. Note that $K$ is a maximal compact subgroup of $G$. The valuation at $T^{-1}$ on $k$ is $v_{\infty}(f/g) = \deg{g}-\deg f$ with $f, \, g \in \F_{q}[T]$, and the norm is given by $|a| = |a|_{\infty} = q^{-v_{\infty}(a)}$ for $a \in k$. For $a \in k_{\infty}$ with $a = \sum_{i = n}^{\infty} a_{i} T^{-i}$, $a_{n} \neq 0$, we have $|a| = q^{-n}$.

\subsection{Bruhat--Tits buildings}\label{Bruhat-Tits-buildings}
The upper half-plane in our case is the homogeneous~space $\mathbb{H} = G/K$. We can endow $\mathbb{H}$ with the structure of the $q+1$ regular tree $X$. Given a graph~$Y$, we write $V(Y)$ for the set of vertices of $Y$ and $E(Y)$ for the set of edges of $Y$. Then we have $\mathbb{H} = V(X)$. The tree $X$ has a natural distance $d$, namely $d(u, v) = 1$ if $u$ and $v$ are adjacent in $X$. This concept plays a key role in the definition of the adjacency operator on $\Gamma \backslash \mathbb{H}$.

We recall the spectral theory of automorphic forms on the Bruhat--Tits building associated to $\PGL_{2}$ of a function field over a finite field, with an emphasis on the spectral expansion and the distinction between discrete and continuous spectra. Efrat~\cite{Efrat1991} observed that the spectrum of the adjacency operator (denoted by $\mathcal{T}$) looks like
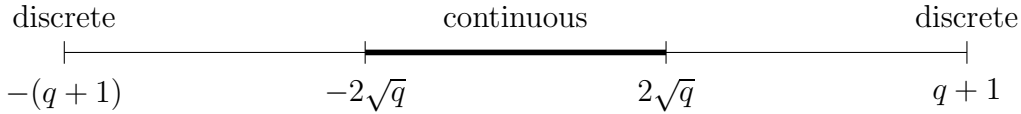
\begin{center}
\begin{tikzpicture}
\draw (0,0) -- (12,0);
\draw [line width=0.75mm] (4,0) -- (8,0);
\foreach \x in {0,4,8,12} \draw (\x, -0.15) -- (\x, 0.15);
\draw(0,-0.2) node[below]{$-(q+1)$};
\draw(4,-0.2) node[below]{$-2 \sqrt{q}$};
\draw(8,-0.2) node[below]{$2 \sqrt{q}$};
\draw(12,-0.2) node[below]{$q+1$};
\draw(0,0.2) node[above]{discrete};
\draw(6,0.2) node[above]{continuous};
\draw(12,0.2) node[above]{discrete};
\end{tikzpicture}
\end{center}
In particular, if $\Gamma = \Gamma(A)$ is the principal congruence subgroup, then Morgenstern~\cite{Morgenstern1995} proved that the eigenvalues of $\mathcal{T}$ except for $\lambda = \pm (q+1)$ satisfy $|\lambda| \leq 2\sqrt{q}$ via the function field analogue of the Ramanujan--Petersson conjecture due to Drinfeld~\cite{Drinfeld1988}. Hence,~$\Gamma \backslash \mathbb{H}$~is a Ramanujan diagram. The interested reader is directed to the work of Morgenstern~\cite{Morgenstern1994} and references therein for Ramanujan diagrams and some relevant information.

\subsection{Eisenstein series}
We follow the notation of Nagoshi~\cite{Nagoshi2001} with minor adjustments. Let $\{\mathfrak{a}, \mathfrak{b}, \mathfrak{c}, \dots \}$ be the set of inequivalent cusps for $\Gamma$. Let $\Gamma_{\mathfrak{a}}$ be the stabiliser of the cusp $\mathfrak{a}$ in $\Gamma$, and let $\sigma_{\mathfrak{a}} \in G$ denote the scaling matrix such that $\sigma_{\mathfrak{a}} \infty = \mathfrak{a}$ and $\sigma_{\mathfrak{a}}^{-1} \Gamma_{\mathfrak{a}} \sigma_{\mathfrak{a}} = \Gamma_{\infty}$. Given $g \in \mathbb{H}$ and $s \in \C$, the Eisenstein series attached to the cusp $\mathfrak{a}$ is defined by
\begin{equation*}
E_{\mathfrak{a}}(g, s) \coloneqq \sum_{\gamma \in \Gamma_{\mathfrak{a}} \backslash \Gamma} 
\psi_{s}(\sigma_{\mathfrak{a}}^{-1} \gamma g),
\end{equation*}
where
\begin{equation*}
\psi_{s}(g) \coloneqq \left(\frac{|\det(g)|}{h((0, 1) g)^{2}} \right)^{s}.
\end{equation*}
It satisfies the relation
\begin{equation*}
(\mathcal{T} E_{\mathfrak{a}})(g, s) = (q^{s}+q^{1-s}) E_{\mathfrak{a}}(g, s).
\end{equation*}
Furthermore, it is invariant under $\Gamma$ and has a Fourier--Whittaker expansion at each cusp.~In the present paper, we focus on the cusp at infinity. For notational convenience, we suppress its dependence in the subscript and write $E(g, s) \coloneqq E_{\infty}(g, s)$.

\subsection{Incomplete Eisenstein series}\label{incomplete-Eisenstein-series}
Let
\begin{equation*}
\Gamma = \Gamma_{0}(A) \coloneqq \left\{g = \begin{pmatrix} a & b \\ c & d \end{pmatrix} \in \PGL_{2}(\F_{q}[T]): c \equiv 0 \tpmod{A} \right\}.
\end{equation*}
An incomplete Eisenstein series $E(g|\psi)$ is defined by
\begin{equation*}
E(g|\psi) \coloneqq \sum_{\gamma \in \Gamma_{\infty} \backslash \Gamma}
\psi \left(\frac{|\det(\gamma g)|}{h((0, 1) \gamma g)^{2}} \right),
\end{equation*}
where $\psi$ is a compactly supported test function on $\R_{> 0}$, namely $\psi \in \mathcal{C}_{0}^{\infty}(\R_{> 0})$. Let $\mathcal{B}(\Gamma \backslash G)$~be the space of smooth and bounded automorphic functions $f$ satisfying $f(\gamma g) = f(g)$ for any $\gamma \in \Gamma$ and $g \in G$. If $\mathcal{E}(\Gamma \backslash G)$ denotes the space spanned by incomplete Eisenstein series, then we are interested in its orthogonal complement in $\mathcal{B}(\Gamma \backslash G)$ with respect to the Petersson inner product, obtaining the following lemma.
\begin{lemma}\label{lem:orthogonal}
The orthogonal complement to $\mathcal{E}(\Gamma \backslash G)$ in $\mathcal B(\Gamma \backslash G)$ is the space of cusp forms.
\end{lemma}

\begin{proof}
Suppose $f \in \mathcal{E}(\Gamma \backslash G)^{\perp}$. By unfolding, the inner product $\langle f, E(\ast|\psi) \rangle$ is equal to
\begin{equation*}
\int_{\Gamma \backslash \mathbb{H}} f(g) \sum_{\gamma \in \Gamma_{\infty} \backslash \Gamma}
\overline{\psi \left(\frac{|\det(\gamma g)|}{h((0, 1) \gamma g)^{2}} \right)} dg
 = \int_{\Gamma_{\infty} \backslash \mathbb{H}} f(g) \overline{\psi(|\det(g)|)} dg,
\end{equation*}
where $g = \begin{psmallmatrix} T^{n} & x \\ 0 & 1 \end{psmallmatrix} = \begin{psmallmatrix} 1 & x \\ 0 & 1 \end{psmallmatrix} \begin{psmallmatrix} T^{n} & 0 \\ 0& 1 \end{psmallmatrix}$ so that $h((0, 1) g) = 1$. Furthermore, the Fourier--Whittaker expansion of $f$ due to Li~\cite{Li1978} has the form
\begin{equation}\label{Fourier-development-of-f}
f(g) = f \left(\begin{pmatrix} T^{n} & x \\ 0 & 1 \end{pmatrix} \right) = \sum_{Q \in \F_{q}[T]} c(n, Q) \chi_{Q}(x),
\end{equation}
where $\chi_{Q}(x) = \chi(Qx)$ is a character of $\F_{q}((T^{-1}))$ trivial on $A$. Hence, we deduce
\begin{equation*}
\langle f, E(\ast|\psi) \rangle
 = \sum_{n \in \Z} \sum_{Q \in \F_{q}[T]} c(n, Q) \overline{\psi(q^{n})} \int_{\F_{q}((T^{-1}))/\F_{q}[T]} \chi_{Q}(x) dx
 = \sum_{n \in \Z} c(n, 0) \overline{\psi(q^{n})}
\end{equation*}
thanks to the orthogonality relation
\begin{equation}\label{orthogonality-of-chi}
\int_{\F_{q}((T^{-1}))/\F_{q}[T]} \chi_{Q}(x) dx = 
	\begin{cases}
	1 & \text{if $Q = 0$},\\
	0 & \text{if $Q \ne 0$}.
	\end{cases}
\end{equation}
Since $\langle f, E(\ast|\psi) \rangle = 0$ for any $\psi$, we have $c(n, 0) = 0$ for all $n \in \Z$. The treatment of other cusps is quite similar. The proof of Lemma~\ref{lem:orthogonal} is complete.
\end{proof}

\subsection{Index}
We compute the index of a Hecke congruence subgroup $\Gamma_{0}(A)$ in $\PGL_{2}(\F_{q}[T])$.
\begin{lemma}\label{index}
The index $m$ of $\Gamma_{0}(A)$ in $\PGL_{2}(\F_{q}[T])$ is given by $m = |A|+1$.
\end{lemma}

\begin{proof}
Let $\pi$ be the natural surjection from $\PGL_{2}(\F_{q}[T])$ onto $\PGL_{2}(\F_{q}[T]/(A))$. We choose representatives so that $\PGL_{2}(\F_{q}[T])$ consists of the matrices of the form $\begin{psmallmatrix} a & b \\ c & 1 \end{psmallmatrix}$ or $\begin{psmallmatrix} a & b \\ 1 & 0 \end{psmallmatrix}$. Let $\pi \left(\begin{psmallmatrix} a & b \\ c & d \end{psmallmatrix} \right) = \begin{psmallmatrix} [a] & [b] \\ [c] & [d] \end{psmallmatrix}$. We first consider the order of $\PGL_{2}(\F_{q}[T]/(A))$, where~the order of the quotient ring $\F_{q}[T]/(A)$ is $|A|$. The number of possibilities of the first row $([a], [b])$ equals $|A|^{2}-1$, because all pairs except $(0, 0)$ are counted. Then for each $([a], [b])$, the number of possibilities of the second row $([c], [d])$ is equal to $|A|$, because $\begin{psmallmatrix} a & b \\ c & 1 \end{psmallmatrix} \ (c \in \F_{q} \setminus \{a/b \})$~and $\begin{psmallmatrix} a & b \\ 1 & 0 \end{psmallmatrix}$ for $[b] \ne 0$, and $\begin{psmallmatrix} a & b \\ c & 1 \end{psmallmatrix} \ (c \in \F_{q})$ for $b = 0$. It now follows that
\begin{equation*}
\#\PGL_{2}(\F_{q}[T]/(A)) = |A|(|A|^{2}-1).
\end{equation*}
This implies that the index of the principal congruence subgroup $\Gamma(A)$ in $\PGL_{2}(\F_{q}[T])$ is
\begin{equation*}
\#(\PGL_{2}(\F_{q}[T])/\operatorname{Ker} \pi) = \#\operatorname{Im} \pi = |A|(|A|^{2}-1).
\end{equation*}
On the other hand, the index of $\Gamma(A)$ in $\Gamma_{0}(A)$ is $|A|(|A|-1)$ since $[a]$ can be~any nonzero element in $\F_{q}$, while $[b]$ can be any element in $\F_{q}$. Hence, the desired index boils down to
\begin{equation*}
m = [\PGL_{2}(\F_{q}[T]): \Gamma_{0}(A)] = \frac{[\PGL_{2}(\F_{q}[T]): \Gamma(A)]}{[\Gamma_{0}(A): \Gamma(A)]}
 = \frac{|A|(|A|^{2}-1)}{|A|(|A|-1)} = |A|+1.
\end{equation*}
The proof of Lemma~\ref{index} is complete.
\end{proof}

\section{Proof of Theorem~\ref{main}}\label{proof-of-theorem-1.4}
This section contains the proof of Theorem~\ref{main}. Since the spectral decomposition shows that $\phi \in L^{2}(\Gamma_{0}(A) \backslash \mathbb{H})$ may be decomposed in terms of Hecke--Maa{\ss} cusp forms $\{u_{j} \}$ and incomplete Eisenstein series $E(g|\psi)$, it suffices to examine their contributions separately.

\subsection{Cuspidal contribution}\label{contribution-of-the-cuspidal-spectrum}
We prove that the cuspidal contribution is negligibly small as $|A|$ tends to infinity.
\begin{theorem}\label{3.1}
Let $u \in L^{2}(\Gamma_{0}(A) \backslash \mathbb{H})$ be a Hecke--Maa{\ss} cusp form. Then
\begin{equation*}
\int_{\Gamma_{0}(A) \backslash \mathbb{H}} u(g) \left|E \left(g, \frac{1}{2}+it \right) \right|^{2} dg \ll |A|^{-1}.
\end{equation*}
\end{theorem}

Given monic irreducible polynomials $P \in \mathbb{F}_{q}[T]$ and a monic polynomial $X \in \mathbb{F}_{q}[T]$, it is convenient to define the M\"{o}bius function by
\begin{equation*}
\mu(X) \coloneqq 
	\begin{cases}
	0 & \text{if $\#\{P: P^{2} \mid X \} > 0$},\\
	1 & \text{if $\#\{P: P \mid X \} \equiv 0 \tpmod{2}$},\\
	-1 & \text{if $\#\{P: P \mid X \} \equiv 1 \tpmod{2}$},
	\end{cases}
\end{equation*}
and Euler's totient function for nonzero $X \in \mathbb{F}_{q}[T]$ by
\begin{equation*}
\varphi(X) \coloneqq \#(\mathbb{F}_{q}[T]/(X))^{\times}.
\end{equation*}
We first calculate the Ramanujan sum over function fields.
\begin{lemma}\label{3.2}
Keep the notation as above. Let $X, Y \in \mathbb{F}_{q}[T]$ be monic, and let $(X, Y)$ be the monic polynomial of highest order dividing both $X$ and $Y$. Then
\begin{equation*}
\sum_{\substack{Y \equiv 1 \tpmod A \\ Y \tpmod{AX} \\ (X, Y) = 1}} \chi_{Q} \left(\frac{Y}{X} \right)
 = \mu \left(\frac{X}{(X, Q)} \right) \varphi \left(\frac{X}{(X, Q)} \right)^{-1} \varphi(X).
\end{equation*}
\end{lemma}

\begin{proof}
If we define
\begin{equation*}
g(X) \coloneqq \sum_{\substack{Y \equiv 1 \tpmod A \\ Y \tpmod{AX}}} \chi_{Q} \left(\frac{Y}{X} \right),
\end{equation*}
then
\begin{equation*}
\chi_{Q} \left(-\frac{1}{X} \right) g(X) 
 = \sum_{\substack{Y \equiv 0 \tpmod A \\ Y \tpmod{AX}}} \chi_{Q} \left(\frac{Y}{X} \right)
 = \sum_{Y \tpmod X} \chi_{Q} \left(\frac{AY}{X} \right)
 = 
	\begin{cases}
	|X| & \text{if $X \mid Q$},\\
	0 & \text{otherwise}.
	\end{cases}
\end{equation*}
On the other hand, we make the changes of variables $X = dX_{1}$ and $Y+1 = dY_{1}$, deducing
\begin{equation}\label{second-way}
g(X) = \sum_{d \mid X} \sum_{\substack{Y \equiv 0 \tpmod A \\ Y \tpmod{AX} \\ (X, Y+1) = d}} \chi_{Q} \left(\frac{Y+1}{X} \right)
 = \sum_{d \mid X} \sum_{\substack{Y_{1} \equiv X_{1} X^{-1} \tpmod A \\ Y_{1} \tpmod{AX_{1}} \\ (X_{1}, Y_{1}) = 1}} 
\chi_{Q} \left(\frac{Y_{1}}{X_{1}} \right).
\end{equation}
If the inner sum on the right-hand side of~\eqref{second-way} is denoted by $C_{X_{1}}(Q)$, then
\begin{equation*}
g(X) = \sum_{d \mid X} C_{\frac{X}{d}}(Q) = \sum_{d \mid X} C_{d}(Q).
\end{equation*}
By M\"{o}bius inversion, we conclude that
\begin{equation}\label{Mobius-inversion}
C_{X}(Q) = \sum_{d \mid X} \mu \left(\frac{X}{d} \right) \chi_{Q} \left(-\frac{1}{d} \right) g(d)
 = \mu \left(\frac{X}{(X, Q)} \right) \varphi \left(\frac{X}{(X, Q)} \right)^{-1} \varphi(X).
\end{equation}
The proof of Lemma~\ref{3.2} is complete.
\end{proof}

We then calculate a Dirichlet series that imitates~\cite[Lemma~2.2]{Koyama2009}.
\begin{lemma}\label{Dirichlet-convolution}
Keep the notation as above. Then
\begin{equation*}
\sum_{X \ne 0} \frac{C_{X}(Q)}{|X|^{s}} = \frac{\sigma_{1-s}(Q)}{\zeta_{\F_{q}[T]}(s)},
\end{equation*}
where
\begin{equation}\label{zeta}
\zeta_{\F_{q}[T]}(s) \coloneqq \sum_{X \ne 0} \frac{1}{|X|^{s}} = \frac{1}{1-q^{1-s}}.
\end{equation}
\end{lemma}

\begin{proof}
The claim is immediate from~\eqref{Mobius-inversion} and routine calculations.
\end{proof}

The following lemma is useful when we calculate the Fourier--Whittaker coefficients of the Eisenstein series.
\begin{lemma}\label{general}
Keep the notation as above. Let $A^{\alpha} \parallel Q$, and let $a = \deg{A}$. Then
\begin{equation*}
\sum_{\substack{X \ne 0 \\ A \mid X}} \frac{C_{X}(Q)}{|X|^{2s}}
 = \frac{1}{\zeta_{\F_{q}[T]}(2s)} \left(\sigma_{1-2s}(Q)-\frac{\sigma_{1-2s}(QA^{-\alpha})}{1-q^{-2as}} \right).
\end{equation*}
\end{lemma}

\begin{proof}
A simple consideration shows that
\begin{equation*}
\sum_{\substack{X \ne 0 \\ A \mid X}} \frac{C_{X}(Q)}{|X|^{2s}}
 = \sum_{X \ne 0} \frac{C_{X}(Q)}{|X|^{2s}}-\sum_{\substack{X \ne 0 \\ A \nmid X}} \frac{C_{X}(Q)}{|X|^{2s}}
 = \frac{1}{\zeta_{\F_{q}[T]}(2s)} \left(\sigma_{1-2s}(Q)-\frac{\sigma_{1-2s}(QA^{-\alpha})}{1-q^{-2as}} \right),
\end{equation*}
as required.
\end{proof}

On the other hand, the following lemma is useful when we calculate the constant term in the Fourier--Whittaker expansion of the Eisenstein series.
\begin{lemma}\label{lemma34}
Keep the notation as above. Then
\begin{equation}\label{constant}
\sum_{\substack{X \ne 0 \\ A \mid X}} \frac{1}{|X|^{2s}} 
\sum_{\substack{Y \equiv 1 \tpmod A \\ Y \tpmod{AX} \\ (X, Y) = 1}} 1
 = \frac{q^{a(1-2s)}}{1-q^{-2as}} \frac{\zeta_{\F_{q}[T]}(2s-1)}{\zeta_{\F_{q}[T]}(2s)}.
\end{equation}
\end{lemma}

\begin{proof}
The inner sum on the left-hand side of~\eqref{constant} can be simplified as
\begin{equation*}
\sum_{\substack{Y \equiv 1 \tpmod A \\ Y \tpmod{AX} \\ (X, Y) = 1}} 1
 = \sum_{\substack{(AX, Y) = 1 \\ Y \tpmod{AX}}} \frac{1}{\varphi(A)}
 = \frac{\varphi(AX)}{\varphi(A)}.
\end{equation*}
If we set $X = AX_{1}$ and $\tilde{\varphi}(X_{1}) = \varphi(A^{2} X_{1})/\varphi(A^{2})$, then
\begin{equation*}
\frac{1}{\varphi(A)} \sum_{\substack{X \ne 0 \\ A \mid X}} \frac{\varphi(AX)}{|X|^{2s}}
 = \frac{\varphi(A^{2})}{\varphi(A)|A|^{2s}} 
\sum_{X_{1} \ne 0} \frac{\tilde{\varphi}(X_{1})}{|X_{1}|^{2s}}
 = q^{a(1-2s)} \prod_{P: \text{ irreducible}} \sum_{k = 0}^{\infty} \frac{\tilde{\varphi}(P^{k})}{|P|^{2ks}},
\end{equation*}
where $a = \deg A$. An explicit computation now yields
\begin{equation*}
\tilde{\varphi}(P^{k}) = 
	\begin{dcases}
	\varphi(P^{k}) = |P|^{k}-|P|^{k-1} & \text{if $P \ne A$ and $k \ne 0$},\\
	\dfrac{\varphi(A^{k+2})}{\varphi(A^{2})} = |A|^{k} & \text{if $P = A$},\\
	1 & \text{if $k = 0$},
	\end{dcases}
\end{equation*}
which implies
\begin{equation*}
q^{a(1-2s)} \left(\sum_{k = 0}^{\infty} \frac{|A|^{k}}{|A|^{2ks}} \right)
\prod_{P \ne A} \left(1+\sum_{k = 1}^{\infty} \frac{|P|^{k}-|P|^{k-1}}{|P|^{2ks}} \right)
 = \frac{q^{a(1-2s)}}{1-q^{-2as}} \frac{\zeta_{\F_{q}[T]}(2s-1)}{\zeta_{\F_{q}[T]}(2s)}.
\end{equation*}
The proof of Lemma~\ref{lemma34} is complete.
\end{proof}

The Fourier--Whittaker expansion of the Eisenstein series due to Li~\cite{Li1978} reads
\begin{equation}\label{FourierE}
E(g, s) = \sum_{\substack{Q \in \F_{q}[T] \\ \text{monic}}} c(n, Q, s) \chi_{Q}(x),
\end{equation}
where
\begin{equation*}
c(n, 0, s) = q^{ns}+q^{n(1- s)+1-a} \frac{1-q^{-2s}}{1-q^{1-2s}} 
\sum_{\substack{X \ne 0 \\ A \mid X}} \frac{1}{|X|^{2s}} 
\sum_{\substack{Y \equiv 1 \tpmod A \\ Y \tpmod{AX} \\ (X, Y) = 1}} 1,
\end{equation*}
$c(n, Q, s) = 0$ if $Q \ne 0$ and $n > a-2-\deg{Q}$, and
\begin{equation*}
c(n, Q, s) = q^{n(1- s)+1-a} \frac{(1-q^{-2s})(1-q^{(d+1)(1-2s)})}{1-q^{1-2s}} 
\sum_{\substack{X \ne 0 \\ A \mid X}} \frac{1}{|X|^{2s}} 
\sum_{\substack{Y \equiv 1 \tpmod A \\ Y \tpmod{AX} \\ (X,Y) = 1}} \chi_{Q} \left(\frac{Y}{X} \right)
\end{equation*}
with $d \coloneqq a-2-\deg{Q}-n$ for $n \leq a-2-\deg{Q}$. Lemmata~\ref{3.2}--\ref{lemma34} imply that
\begin{equation}\label{cn0s}
c(n, 0, s) = q^{ns}+\frac{q^{n(1- s)+1-2as}}{1-q^{-2as}},
\end{equation}
and if $Q \ne 0$ and $n \leq a-2-\deg{Q}$, then
\begin{equation}\label{cnQs}
c(n, Q, s) = q^{n(1-s)+1-a}(1-q^{-2s})(1-q^{(a-1-\deg{Q}-n)(1-2s)}) 
\left(\sigma_{1-2s}(Q)-\frac{\sigma_{1-2s}(QA^{-\alpha})}{1-q^{-2as}} \right).
\end{equation}

On the other hand, the Fourier--Whittaker expansion of cusp forms due to Dutta Gupta~\cite{DuttaGupta1997} reads\footnote{Dutta Gupta~\cite[Page~105]{DuttaGupta1997} assumes that $q \equiv 1 \pmod 4$ is a prime for simplicity, but the formul{\ae}~for the Fourier--Whittaker coefficients are still valid for any prime $q > 3$. In fact, such a restriction is not imposed in his statement of the Fourier--Whittaker expansion on~\cite[Page~103]{DuttaGupta1997}.}
\begin{equation}\label{Fourieru}
u(g) = \sum_{Q \ne 0} c(n, Q) \chi_{Q}(x), \qquad c(n, Q) = \frac{c(Q)}{\sqrt{|Q|}} W_{u}(QT^{n+2}),
\end{equation}
where for any unit $\epsilon \in r_{\infty}^{\ast}$, $W_{u}$ denotes the Whittaker function given by
\begin{equation}\label{Wu}
W_{u}(\epsilon T^{-\beta}) \coloneqq 
	\begin{dcases}
	\dfrac{q^{it(\beta+1)}-q^{-it(\beta+1)}}{q^{it}-q^{-it}} & \text{if $\beta \geq 0$},\\
	0 & \text{if $\beta < 0$},
	\end{dcases}
\end{equation}
where $\frac{1}{4}+t^{2}$ denotes the eigenvalue of the infinite Hecke operator associated to $u$. We recall the assumption $t \in \R^{\times}$. If $u$ is a newform, then Goss~\cite{Goss1980} and Li and Meemark~\cite{LiMeemark2008} proved that the normalised Fourier--Whittaker coefficients $\tilde{c}(Q) = c(Q)/c(1)$ are completely multiplicative, which ensures an Euler product of degree $1$:
\begin{equation*}
L(s, u) = \sum_{\substack{Q \in \F_{q}[T] \\ \text{monic}}} \frac{\tilde{c}(Q)}{|Q|^{s}}
 = \prod_{\substack{P \in \F_{q}[T] \\ \text{irreducible}}} \left(1-\frac{\tilde{c}(P)}{|P|^{s}} \right)^{-1}.
\end{equation*}

If $u$ is an oldform, then there exists a Hecke--Maa{\ss} cusp form $v(g)$ for $\Gamma_{0}(1)$ such that~either $u(g) = v(g)$ or $u(g) = v \left(\begin{psmallmatrix} A & 0 \\ 0 & 1 \end{psmallmatrix} g \right)$. In the former case, the normalised Fourier--Whittaker coefficients $\tilde{c}(Q)$ are similarly defined, and $L(s, u)$ has an Euler product of degree $1$. In the latter case, if the Fourier--Whittaker coefficients of $v(g)$ are denoted by $c^{\ast}(Q)$, then
\begin{equation}\label{Fourierv}
v(g) = \sum_{\substack{Q \in \F_{q}[T] \\ Q \ne 0}} \frac{c^{\ast}(Q)}{\sqrt{|Q|}} W_{u}(QT^{n+2}) \chi_{Q}(x).
\end{equation}
Since $\begin{psmallmatrix} A & 0 \\ 0 & 1 \end{psmallmatrix} \begin{psmallmatrix} T^{n} & x \\ 0 & 1 \end{psmallmatrix} = \begin{psmallmatrix} AT^{n} & Ax \\ 0 & 1 \end{psmallmatrix}$, we arrive at the expression
\begin{equation*}
u(g) = \sum_{\substack{Q \in \F_{q}[T] \\ Q \ne 0}} \frac{c^{\ast}(Q)}{\sqrt{|Q|}} W_{u}(QT^{n+a+2}) \chi_{QA}(x)
 = \sum_{\substack{Q \in \F_{q}[T] \\ A \mid Q \ne 0}} \frac{c^{\ast}(Q/A)}{\sqrt{|Q/A|}} W_{u}(QT^{n+2}) \chi_{Q}(x).
\end{equation*}
This yields the relation
\begin{equation}\label{relation}
c(Q) = 
	\begin{dcases}
	\sqrt{|A|} \, c^{\ast} \left(\frac{Q}{A} \right) & \text{if $A \mid Q$},\\
	0 & \text{otherwise}.
	\end{dcases}
\end{equation}

The following lemma is an analogue of~\cite[Lemma~2.4]{Koyama2009}; see also~\cite{KanekoKoyama2020} for corrections.
\begin{lemma}\label{lemm:NewOld}
Keep the notation as above. If $u$ is a newform for $\Gamma_{0}(A)$ or a Hecke--Maa{\ss}~cusp form for $\Gamma_{0}(1)$, then
\begin{equation}
\sum_{\substack{Q \ne 0 \\ \text{monic}}}
\frac{\tilde{c}(Q) \sigma_{\nu}(Q)}{|Q|^{s}} = \frac{L(s, u)L(s-\nu, u)}{\zeta_{\F_{q}[T]}(2s-\nu)}.
\end{equation}
If $u$ is an oldform for $\Gamma_{0}(A)$ and $u(g) = v \left(\begin{psmallmatrix} A & 0 \\ 0 & 1 \end{psmallmatrix} g \right)$, then
\begin{equation}
\sum_{\substack{Q \ne 0 \\ \text{monic}}} \frac{c(Q) \sigma_{\nu}(Q)}{|Q|^{s}}
 = c^{\ast}(1) |A|^{\frac{1}{2}-s} \frac{1+|A|^{\nu}-\tilde{c}^{\ast}(A) |A|^{\nu-s}}{1-|A|^{\nu-2s}} 
\frac{L(s, v)L(s-\nu, v)}{\zeta_{\F_{q}[T]}(2s-\nu)}.
\end{equation}
\end{lemma}

\begin{proof}
The first claim follows from
\begin{equation*}
\sum_{\substack{Q \ne 0 \\ \text{monic}}} \frac{\tilde{c}(Q) \sigma_{\nu}(Q)}{|Q|^{z}}
 = \prod_{\substack{P: \text{ monic} \\ \text{irreducible}}} \sum_{k = 0}^{\infty} \frac{\tilde{c}(P^{k}) \sigma_{\nu}(P^{k})}{|P|^{ks}}\\
 = \frac{L(s, u) L(s-\nu, u)}{\zeta_{\F_{q}[T]}(2s-\nu)}.
\end{equation*}
For the second claim, the relation~\eqref{relation} implies
\begin{equation}\label{oldformsum}
\sum_{\substack{Q \ne 0 \\ \text{monic}}} \frac{c(Q) \sigma_{\nu}(Q)}{|Q|^{s}}
 = \sum_{\substack{Q \ne 0 \\ \text{monic}}} \frac{c(AQ) \sigma_{\nu}(AQ)}{|AQ|^{s}}
 = |A|^{\frac{1}{2}-s} c^{\ast}(1) \sigma_{\nu}(A) \prod_{\substack{P: \text{ monic} \\ \text{irreducible}}} \sum_{k = 0}^{\infty} \frac{\tilde{c}^{\ast}(P^{k}) \tilde{\sigma}_{\nu}(AP^{k})}{|P|^{ks}},
\end{equation}
where $\tilde{\sigma}_{\nu}(AQ) = \sigma_{\nu}(AQ)/\sigma_{\nu}(A)$ is multiplicative in $Q$. Note that
\begin{equation}\label{tildesigma}
\tilde{\sigma}_{\nu}(AP^{k}) = 
	\begin{dcases}
	\dfrac{1-|A|^{(k+2) \nu}}{1-|A|^{2\nu}} & \text{if $P = A$},\\
	\dfrac{1-|P|^{(k+1) \nu}}{1-|P|^{\nu}} & \text{if $P\ne A$}.
	\end{dcases}
\end{equation}
Substituting~\eqref{tildesigma} into~\eqref{oldformsum} and following the proof of~\cite[Lemma 2.4]{Koyama2009}, we obtain the second claim.
\end{proof}

\begin{proof}[Proof of Theorem~\ref{3.1}.]
Unfolding gives
\begin{equation}\label{cusp1}
\int_{\Gamma \backslash \mathbb{H}} u(g)|E(g, s)|^{2} dg
 = \int_{\Gamma \backslash \mathbb{H}} u(g) \overline{E(g, s)}
\sum_{\gamma \in \Gamma_{\infty} \backslash \Gamma} \psi_{s}(\gamma g) dg
 = \int_{\Gamma_{\infty} \backslash \mathbb{H}} u(g) \overline{E(g, s)} q^{ns} dg.
\end{equation}
Substituting~\eqref{FourierE} and~\eqref{Fourieru} into~\eqref{cusp1} and using~\eqref{orthogonality-of-chi},~\eqref{cn0s}, and~\eqref{cnQs}, we derive
\begin{multline}\label{sumQ}
\sum_{n \in \Z} \sum_{Q \ne 0} \frac{c(Q)}{\sqrt{|Q|}} q^{ns} c(n, Q, s) W_{u}(QT^{n+2})\\
 = q^{1-a}(1-q^{-2s}) \sum_{Q \ne 0} \frac{c(Q)}{\sqrt{|Q|}} 
\left(\sigma_{1-2s}(Q)-\frac{\sigma_{1-2s}(QA^{-\alpha})}{1-q^{-2as}} \right)\\
\times \sum_{n = -\infty}^{a-2-\deg{Q}} q^{n}(1-q^{(a-1-\deg{Q}-n)(1-2s)}) W_{u}(QT^{n+2}).
\end{multline}
It follows from~\eqref{Wu} with $\beta = -n-2-\deg{Q}$ that the first term in the sum over $n$ contributes
\begin{equation*}
\sum_{n = -\infty}^{a-2-\deg{Q}} q^{n} \frac{q^{-it(n+1+\deg{Q})}-q^{it(n+1+\deg{Q})}}{q^{it}-q^{-it}}
 = \frac{q^{a-2-\deg Q}}{q^{it}-q^{-it}} \left(\frac{q^{-(a-1)it}}{1-q^{-(1-it)}}-\frac{q^{(a-1)it}}{1-q^{-(1+it)}} \right),
\end{equation*}
while the the second term contributes
\begin{multline*}
\sum_{n = -\infty}^{-2-\deg{Q}} q^{2ns+(1-2s)(a-1-\deg{Q})} \frac{q^{-it(n+1+\deg{Q})}-q^{it(n+1+\deg{Q})}}{q^{it}-q^{-it}}\\
 = \frac{q^{a-1-\deg{Q}}}{q^{it}-q^{-it}} \left(\frac{q^{-iat}}{1-q^{2s+it}}-\frac{q^{iat}}{1-q^{2s-it}} \right).
\end{multline*}
Lemma~\ref{lemm:NewOld} implies that the sum over $Q$ in~\eqref{sumQ} is bounded. To conclude the proof, recall that $a = \deg A$ and thus the right-hand side of~\eqref{sumQ} is bounded by $O(|A|^{-1})$.
\end{proof}

\subsection{Eisenstein contribution}\label{contribution-of-the-Eisenstein-spectrum}
In this section, we analyse the inner product
\begin{equation*}
\mathcal{I} \coloneqq \langle |E|^{2}, E(\ast|\psi) \rangle
 = \int_{\Gamma_{0}(A) \backslash \mathbb{H}} E(g|\psi) \left|E \left(g, \frac{1}{2}+it \right) \right|^{2} dg,
\end{equation*}
where $\psi \in \mathcal{C}_{0}^{\infty}(\R_{> 0})$. The Mellin transform of $\psi$ is defined by
\begin{equation*}
H(s) \coloneqq \sum_{n \in \Z} \psi(q^{n}) q^{-ns},
\end{equation*}
and Mellin inversion gives
\begin{equation*}
\psi(q^{n}) = \log q \int_{-\frac{\pi i}{\log q}}^{\frac{\pi i}{\log q}} H(s) q^{ns} \frac{ds}{2\pi i},
\end{equation*}
where the contour is along the imaginary axis. Unfolding then yields
\begin{align*}
\mathcal{I} &= \int_{\Gamma_{0}(A) \backslash \mathbb{H}} \sum_{\gamma \in \Gamma_{\infty} \backslash \Gamma_{0}(A)}
\psi \left(\frac{|\det(\gamma g)|}{h((0,1) \gamma g)^{2}} \right) \left|E \left(g, \frac{1}{2}+it \right) \right|^{2} dg\\
& = \log q \int_{-\frac{\pi i}{\log q}}^{\frac{\pi i}{\log q}} H(s) \int_{\Gamma_{\infty} \backslash \mathbb{H}}
\left(\frac{|\det(g)|}{h((0,1) g)^{2}} \right)^{s} \left|E \left(g, \frac{1}{2}+it \right) \right|^{2} dg \frac{ds}{2\pi i}.
\end{align*}
If $g = \begin{psmallmatrix} T^{n} & x \\ 0 & 1 \end{psmallmatrix}$, then $h((0, 1) g) = 1$. Hence, the problem reduces to the estimation of
\begin{equation*}
\mathcal{I} = \log q \int_{-\frac{\pi i}{\log q}}^{\frac{\pi i}{\log q}} H(s) \sum_{n \in \Z} q^{ns}
\int_{\F_{q}((T^{-1}))/\F_{q}[T]} \left|E \left(g, \frac{1}{2}+it \right) \right|^{2} dx \frac{ds}{2\pi i}.
\end{equation*}
The Fourier--Whittaker expansion~\eqref{FourierE} now leads to the decomposition
\begin{equation}\label{I}
\mathcal{I} = \mathcal{I}_{1}+\mathcal{I}_{2},
\end{equation}
where
\begin{align*}
\mathcal{I}_{1} &\coloneqq \log q \int_{-\frac{\pi i}{\log q}}^{\frac{\pi i}{\log q}} H(s) 
\sum_{n \in \Z} q^{ns} \left|c \left(n, 0, \frac{1}{2}+it \right) \right|^{2} \frac{ds}{2\pi i},\\
\mathcal{I}_{2} &\coloneqq \log q \int_{-\frac{\pi i}{\log q}}^{\frac{\pi i}{\log q}} H(s) 
\sum_{n \in \Z} \sum_{\substack{Q \in \F_{q}[T] \\ \text{monic}}} 
q^{ns} \left|c \left(n, Q, \frac{1}{2}+it \right) \right|^{2} \frac{ds}{2\pi i}.
\end{align*}
We first prove that $\mathcal{I}_{1}$ is negligibly small.
\begin{lemma}\label{computation-of-I_1}
Keep the notation as above. Then
\begin{equation*}
\mathcal{I}_{1} \ll 1.
\end{equation*}
\end{lemma}

\begin{proof}
It follows from~\eqref{cn0s} that
\begin{equation*}
\left|c \left(n, 0, \frac{1}{2}+it \right) \right|^{2} \leq q^{n}+\frac{q^{n+2-2a}}{|1-q^{-a(1+2it)}|^{2}}.
\end{equation*}
The contribution of the first term dominates that of the second term. Since the sum over~$n$~is bounded, the proof of Lemma~\ref{computation-of-I_1} is complete.
\end{proof}

We now estimate $\mathcal{I}_{2}$ and prove that the contribution of $\mathcal{I}_{2}$ dominates that of $\mathcal{I}_{1}$.
\begin{lemma}\label{computation-of-I_2}
Keep the notation as above. Then
\begin{equation*}
\mathcal{I}_{2} = \frac{(1+q^{-1})}{2\log q} \log|A| 
\int_{\Gamma_{0}(1) \backslash \mathbb{H}} E(g|\psi) dg+O(|A|^{-1}).
\end{equation*}
\end{lemma}

\begin{proof}
It follows from~\eqref{cnQs} that
\begin{equation*}\label{cnq1/2+it}
\left|c \left(n, Q, \frac{1}{2}+it \right) \right|^{2}
 = q^{n+2-2a} |1-q^{-1-2it}|^{2} |1-q^{-2it(a-1-\deg{Q}-n)}|^{2} 
\left|\sigma_{-2it}(Q)-\frac{\sigma_{-2it}(QA^{-\alpha})}{1-q^{-a(1+2it)}} \right|^{2},
\end{equation*}
where $Q \ne 0$ and $n \leq a-2-\deg{Q}$. 
Henceforth, we focus on the analysis of the first~term $\sigma_{-2it}(Q)$ since the second term contributes negligibly due to the extra factor $A^{-\alpha}$, obtaining
\begin{multline}\label{suminI2}
\sum_{n \in \Z} \sum_{\substack{Q \in \F_{q}[T] \\ \text{monic}}} q^{ns} \left|c \left(n, Q, \frac{1}{2}+it \right) \right|^{2}
 = q^{2-2a} |1-q^{-1-2it}|^{2} \sum_{\substack{Q \in \F_{q}[T] \\ \text{monic}}} |\sigma_{2it}(Q)|^{2}\\
\times \sum_{n = -\infty}^{a-2-\deg{Q}} q^{n(s+1)}(2-q^{-2it(a-1-\deg{Q}-n)}-q^{2it(a-1-\deg{Q}-n)})+(\cdots),
\end{multline}
where the ellipsis hides an admissible error term. A brute force computation shows that the sum over $n$ in~\eqref{suminI2} is equal to
\begin{equation*}
\left(\frac{|A|}{q^{2}|Q|} \right)^{s+1} 
\left(\frac{2}{1-q^{-(s+1)}}-\frac{q^{-2it}}{1-q^{-(s+1+2it)}}-\frac{q^{2it}}{1-q^{-(s+1-2it)}} \right).
\end{equation*}
Moreover, the function field analogue of Ramanujan's identity~\cite{Ramanujan1916} (see~\cite{GunMurty2016} for example) reads
\begin{align*}
\sum_{\substack{Q \in \F_{q}[T] \\ \text{monic}}} \frac{|\sigma_{2it}(Q)|^{2}}{|Q|^{s+1}}
& = \frac{\zeta_{\F_{q}[T]}(s+1)^{2} \zeta_{\F_{q}[T]}(s+1+2it) \zeta_{\F_{q}[T]}(s+1-2it)}{\zeta_{\F_{q}[T]}(2s+2)}\\
& = \frac{1-q^{-2s-1}}{(1-q^{-s})^{2}(1-q^{-s+2it})(1-q^{-s-2it})}.
\end{align*}
Hence, the integral $\mathcal{I}_{2}$ boils down to
\begin{multline}\label{I2}
q^{2-2a} |1-q^{-1-2it}|^{2} \log q \int_{-\frac{\pi i}{\log q}}^{\frac{\pi i}{\log q}} H(s)
\left(\frac{|A|}{q^{2}} \right)^{s+1} \frac{1-q^{-2s-1}}{(1-q^{-s})^{2}(1-q^{-s+2it})(1-q^{-s-2it})}\\
\times \left(\frac{2}{1-q^{-(s+1)}}-\frac{q^{-2it}}{1-q^{-(s+1+2it)}}-\frac{q^{2it}}{1-q^{-(s+1-2it)}} \right) \frac{ds}{2\pi i}.
\end{multline}
Note that the factor
\begin{equation*}
\frac{1}{(1-q^{-s})^{2}} = \frac{1}{(s \log q)^{2}}+\frac{1}{s \log q}+O(1)
\end{equation*}
has a double pole at $s = 0$. If the integrand in~\eqref{I2} is denoted by $G(s) H(s)(1-q^{-s})^{-2}$, then the residue at $s = 0$ is
\begin{equation}\label{residue-at-0}
\frac{G(0) H(0)}{(\log q)^{2}} \left(\frac{G^{\prime}(0)}{G(0)}+\frac{H^{\prime}(0)}{H(0)}+\log q \right).
\end{equation}
The definition of $G(s)$ gives
\begin{equation}\label{G0}
G(0) = \frac{|A|(1-q^{-1})}{q^{2} |1-q^{2it}|^{2}} \left(\frac{2}{1-q^{-1}}-\frac{q^{-2it}}{1-q^{-1-2it}}-\frac{q^{2it}}{1-q^{-1+2it}} \right)
 = \frac{|A|(1+q^{-1})}{q^{2}|1-q^{-1-2it}|^{2}}
\end{equation}
and
\begin{equation*}
\frac{G^{\prime}(0)}{G(0)} = \log|A|+O(1),
\end{equation*}
yielding
\begin{equation*}
\mathop{\mathrm{res}}_{s = 0} \left(\frac{G(s) H(s)}{(1-q^{-s})^{2}} \right)
 = \frac{|A| \log|A|(1+q^{-1})H(0)}{(q \log q)^{2}|1-q^{-1-2it}|^{2}}+O(|A|).
\end{equation*}
On the other hand, the definition of $H(s)$ gives
\begin{equation*}
H(0) = \sum_{n \in \Z} \psi(q^{n})
 = \int_{\Gamma_{\infty} \backslash \mathbb{H}} \psi \left(\frac{|\det(g)|}{h((0,1) g)^{2}} \right) dg
 = \int_{\Gamma_{0}(A) \backslash \mathbb{H}} E(g|\psi) dg = m \int_{\Gamma_{0}(1) \backslash \mathbb{H}} E(g|\psi) dg,
\end{equation*}
where we relate the inner product of level $A$ to that of level $1$ by folding the integration.~By Lemma~\ref{index}, we conclude that
\begin{equation}\label{res}
\mathop{\mathrm{res}}_{s = 0} \left(\frac{G(s) H(s)}{(1-q^{-s})^{2}} \right)
 = \frac{|A|^{2} \log|A|(1+q^{-1})}{(q \log q)^{2}|1-q^{-1-2it}|^{2}} 
\int_{\Gamma_{0}(1) \backslash \mathbb{H}} E(g|\psi) dg+O(|A|),
\end{equation}
where the contribution of $H^{\prime}(s)/H(s)$ in~\eqref{residue-at-0} is bounded by $O(|A|)$, since $H(s)$ is entire in $s$. The contribution of the simple poles at $s = \pm 2it$ is negligibly small as in~\cite{LuoSarnak1995}.

To examine~\eqref{I2}, we consider the integral over a rectangle $\mathcal{C} = \mathcal{C}_{0} \cup \mathcal{C}_{1} \cup \mathcal{C}_{2} \cup \mathcal{C}_{3}$, where
\begin{alignat*}{2}
\mathcal{C}_{0} &: -\frac{\pi i}{\log q} \longrightarrow \frac{\pi i}{\log q},
&& \mathcal{C}_{1}: \frac{\pi i}{\log q} \longrightarrow R+\frac{\pi i}{\log q},\\
\mathcal{C}_{2} &: R+\frac{\pi i}{\log q} \longrightarrow R-\frac{\pi i}{\log q}, \qquad 
&& \mathcal{C}_{3}: R-\frac{\pi i}{\log q} \longrightarrow -\frac{\pi i}{\log q},
\end{alignat*}
where $R < 0$. Because
\begin{equation*}
H \left(x-\frac{\pi i}{\log q} \right) = H \left(x+\frac{\pi i}{\log q} \right), \qquad 
G \left(x-\frac{\pi i}{\log q} \right) = G \left(x+\frac{\pi i}{\log q} \right),
\end{equation*}
we obtain
\begin{equation}\label{C13}
\int_{\mathcal{C}_{1} \cup \mathcal{C}_{3}} \frac{G(s) H(s)}{(1-q^{-s})^{2}} \frac{ds}{2\pi i} = 0.
\end{equation}
The integral over $\mathcal{C}_{2}$ decays rapidly as $R \to -\infty$. It thus remains to analyse the integral~over $\mathcal{C}_{0}$. Since there is a pole at $s = 0$, we introduce a small half circle of the shape
\begin{equation*}
\mathcal{C}(\delta) \coloneqq \{z \in \C: |z| = \delta, \, \Re(z) > 0 \}.
\end{equation*}
It now follows from~\eqref{res} and~\eqref{C13} that
\begin{multline*}
\int_{\mathcal{C}(\delta)} \frac{G(s) H(s)}{(1-q^{-s})^{2}} \frac{ds}{2\pi i}
 + \int_{-\frac{\pi i}{\log q}}^{-\delta i} \frac{G(s) H(s)}{(1-q^{-s})^{2}} \frac{ds}{2\pi i}
 + \int_{\delta i}^{\frac{\pi i}{\log q}} \frac{G(s) H(s)}{(1-q^{-s})^{2}} \frac{ds}{2\pi i}\\
 = \frac{|A|^{2} \log|A|(1+q^{-1})}{(q \log q)^{2} |1-q^{-1-2it}|^{2}} 
\int_{\Gamma_{0}(1) \backslash \mathbb{H}} E(g|\psi) dg+O(|A|).
\end{multline*}
On the other hand, a straightforward computation gives
\begin{align*}
\int_{\mathcal{C}(\delta)} \frac{G(s) H(s)}{(1-q^{-s})^{2}} \frac{ds}{2\pi i}
 = \frac{|A|^{2} \log|A|(1+q^{-1})}{2(q \log q)^{2} |1-q^{-1-2it}|^{2}} 
\int_{\Gamma_{0}(1) \backslash \mathbb{H}} E(g|\psi) dg+O(|A|).
\end{align*}
It now follows that
\begin{align*}
\int_{\mathcal{C}_{0}} \frac{G(s) H(s)}{(1-q^{-s})^{2}} \frac{ds}{2\pi i}
& = \lim_{\delta \to 0} \left(\int_{-\frac{\pi i}{\log q}}^{-\delta i} \frac{G(s) H(s)}{(1-q^{-s})^{2}} \frac{ds}{2\pi i}
 + \int_{\delta i}^{\frac{\pi i}{\log q}} \frac{G(s) H(s)}{(1-q^{-s})^{2} } \frac{ds}{2\pi i} \right)\\
& = \frac{|A|^{2} \log|A|(1+q^{-1})}{2(q \log q)^{2} |1-q^{-1-2it}|^{2}} 
\int_{\Gamma_{0}(1) \backslash \mathbb{H}} E(g|\psi) dg+O(|A|).
\end{align*}
Putting everything together completes the proof of Lemma~\ref{computation-of-I_2}.
\end{proof}

To conclude our analysis, we combine Lemmata~\ref{computation-of-I_1} and~\ref{computation-of-I_2} and~\eqref{I} to deduce
\begin{equation*}
\mathcal{I} = \frac{(1+q^{-1})}{2\log q} \log|A| \int_{\Gamma_{0}(1) \backslash \mathbb{H}} E(g|\psi) dg+O(1).
\end{equation*}
Theorem~\ref{main} is now a straightforward consequence of the standard approximation argument of Luo--Sarnak~\cite[Page~217]{LuoSarnak1995}. \qed


\providecommand{\bysame}{\leavevmode\hbox to3em{\hrulefill}\thinspace}
\providecommand{\MR}{\relax\ifhmode\unskip\space\fi MR }
\providecommand{\MRhref}[2]{%
  \href{http://www.ams.org/mathscinet-getitem?mr=#1}{#2}
}
\providecommand{\Zbl}{\relax\ifhmode\unskip\space\thinspace\fi Zbl }
\providecommand{\MR}[2]{%
  \href{https://zbmath.org/?q=an:#1}{#2}
}
\providecommand{\doi}{\relax\ifhmode\unskip\space\thinspace\fi DOI }
\providecommand{\MR}[2]{%
  \href{https://doi.org/#1}{#2}
}
\providecommand{\SSNI}{\relax\ifhmode\unskip\space\thinspace\fi ISSN }
\providecommand{\MR}[2]{%
  \href{#1}{#2}
}
\providecommand{\ISBN}{\relax\ifhmode\unskip\space\thinspace\fi ISBN }
\providecommand{\MR}[2]{%
  \href{#1}{#2}
}
\providecommand{\arXiv}{\relax\ifhmode\unskip\space\thinspace\fi arXiv }
\providecommand{\MR}[2]{%
  \href{#1}{#2}
}
\providecommand{\href}[2]{#2}

\end{document}